\documentclass[graybox]{svmult}

\usepackage{type1cm}        
\usepackage{makeidx}         
\usepackage{graphicx}        
\usepackage{multicol}        
\usepackage[bottom]{footmisc}
\usepackage{bbm}
\usepackage{mathtools}

\usepackage{newtxtext}       %
\usepackage[varvw]{newtxmath}       

\mathtoolsset{showonlyrefs}

\newcommand{\R}{\mathbb{R}}
\newcommand{\C}{\mathbb{C}}
\newcommand{\N}{\mathbb{N}}

\newcommand{\cD}{\mathcal{D}}
\newcommand{\cF}{\mathcal{F}}
\newcommand{\cJ}{\mathcal{J}}
\newcommand{\cK}{\mathcal{K}}

\newcommand{\cQ}{\mathcal{Q}}
\newcommand{\cI}{\mathcal{I}}
\newcommand{\sph}[2]{S(#1,#2)}
\newcommand{\ball}[2]{B(#1,#2)}
\newcommand{\sect}[1]{T_{#1}}
\newcommand{\supp}{\mathrm{supp}}

\newcommand{\de}{\mathrm{d}}

\DeclareMathAlphabet{\mathcal}{OMS}{zplm}{m}{n}

\makeindex             

\begin{document}

\title*{H$^1$ and BMO spaces for exponentially decreasing measures on homogeneous trees}
\author{Matteo Monti}
\institute{Matteo Monti \at{Dipartimento di Scienze Matematiche “Giuseppe Luigi Lagrange”, Dipartimento di Eccellenza 2018-2022, Politecnico di Torino, Corso Duca degli Abruzzi 24, 10129, Torino, Italy, email: \mbox{matteo.monti@polito.it}.} }
%
%
\maketitle

\abstract*{We consider a family of measures on a $q$-homogeneous tree that decrease exponentially with respect to the distance from the origin. Such measures are doubling with respect to the Gromov distance. We define atomic Hardy and BMO spaces for that measures, and we prove interpolation results regarding such spaces. As a consequence we have boundedness results for integral operators involving Hardy, BMO, and $L^p$ spaces.}

\abstract{We consider a family of measures on a $q$-homogeneous tree that decrease exponentially with respect to the distance from the origin. Such measures are doubling with respect to the Gromov distance. We define atomic Hardy and BMO spaces for that measures, and we prove interpolation results regarding such spaces. As a consequence we have boundedness results for integral operators involving Hardy, BMO, and $L^p$ spaces.}

\let\thefootnote\relax\footnotetext{This work is partially supported by the project "Harmonic analysis on continuous and discrete structures" funded by Compagnia di San Paolo 
	(Cup E13C21000270007). Furthermore, I am member of the Gruppo Nazionale per l’Analisi Matematica, la Probabilità e le loro Applicazioni (GNAMPA) of the Istituto Nazionale di Alta Matematica (INdAM). }
\let\thefootnote\relax\footnotetext{Keywords: homogeneous trees; Calder\'on-Zygmund theory; atomic Hardy spaces, interpolation.}
\let\thefootnote\relax\footnotetext{Mathematics Subject Classification (2010): 05C05; 30H10; 42B20.}

	\section*{Introduction}
	This chapter is the natural continuation of a joint work with F.~De~Mari and M.~Vallarino~\cite{dmv}, where we study the harmonic Bergman spaces on homogeneous trees associated to a certain class of measures: the exponentially decreasing measures with respect to the distance from the origin. In particular, we show that a Calder\'on-Zygmund decomposition is possible in that setting. The aim of this work is to use that decomposition to obtain interpolation results on atomic Hardy spaces and on bounded mean oscillation spaces (BMO in what follows), and consequent boundedness results for integral operators. 
	
	The Calder\'on-Zygmund decomposition, as well as Hardy and BMO spaces, were introduced for functions on $\R^d$ with respect to the Lebesgue measure~\cite{coifweiss},~\cite{SteinHA}. The Lebesgue measure is doubling with respect to the Euclidean distance. For this reason, several generalizations of this theory have been realized in doubling settings, see for example~\cite{badrruss} and~\cite{russ} for an analysis on graphs in a doubling context.
	The doubling condition is not necessary and it can be weakened by considering locally doubling measures, see~\cite{cmm} and~\cite{cmmfinite}.
	
	On homogeneous trees, the boundedness of singular integrals associated with the combinatorial Laplacian has been investigated in~\cite{cms}, while Celotto and Meda~\cite{celottomeda} studied various Hardy spaces in this context.
	In~\cite{arditti1} and~\cite{arditti2}, Arditti, Tabacco and Vallarino analyze Hardy and BMO spaces for a sort of level measure, a measure having the horocyclic index (with respect to a fixed boundary point) as density. In~\cite{lstv}, the same theory is developed for a more general class of measures, called flow measures. All such measures are not doubling, but only locally doubling. The choice of the distance is not always canonical: the measures we consider here and in~\cite{dmv} are doubling with respect to the Gromov distance but not to the usual graph distance.

	Let $X$ be a $q$-homogeneous tree. The class of measures we consider is formed by the measures $\mu_\alpha(x):=q^{-\alpha|x|}$, $\alpha>1$, where $|x|$ denotes the distance of the vertex $x$ from the origin. A Calder\'on-Zygmund decomposition for functions in $L^1(\mu_\alpha)$ is provided in Proposition~\ref{caldzyg}; it is based on the balls of the Gromov metric, that are essentially sectors of the tree. The measures $\mu_\alpha$ are the analogous of the measures $(1-|x+iy|^2)^{\alpha-2}\de x\de y$, $\alpha>1$, on the hyperbolic disk. In this sense,~\cite{dmv} can be read as a discrete counterpart of~\cite{denghuang} on the hyperbolic disk, where similar results are obtained for Bergman measures.
	
	Following the classical theory (see for example~\cite{coifweiss}), we define atomic Hardy spaces $H^{p,1}$ and bounded mean oscillation spaces ${\rm BMO}_r$, for $1<p\leq\infty$ and $1\leq r<\infty$. It is proved that ${\rm BMO}_{p'}$ characterizes the dual space of $H^{1,p}$ (for $1<p\leq \infty$) and that all the $H^{1,p}$ (and then all the ${\rm BMO}_{p'}$) are the equal as vector spaces with equivalent norms. Hence we put $H^1:=H^{1,\infty}$ and ${\rm BMO}:={\rm BMO}_1$. 
	
	In Section~\ref{secinter}, we prove Theorem~\ref{thmintcompl} and Corollary~\ref{corintcom}, our main results on complex interpolation involving $H^1$, ${\rm BMO}$ and $L^p(\mu_\alpha)$. In particular we show that, for every $1<p<\infty$, $L^p(\mu_\alpha)$ is a complex interpolation space between $L^{1}(\mu_\alpha)$ and ${\rm BMO}$.
	By duality, we have that $L^p(\mu_\alpha)$ is also an interpolation space between $H^1$ and $L^{\infty}(\mu_\alpha)$, and then between $H^1$ and ${\rm BMO}$.
	The proof is based on a good lambda inequality presented in Proposition~\ref{propgoodlambda}.
	
	From results of the form of Corollary~\ref{corintcom}, boundedness results for integral operators classically follow. We resume them in Theorem~\ref{theobound}. We need a reformulation of the classical H\"ormander's condition for a kernel $K\colon X\times X\to\C$: if we denote by $\sect{v}$ the sector of $v\in X$, then the condition reads
	\begin{equation}\label{hormintro}
			\sup_{v\in X\setminus\{o\}}\sup_{x,y\in \sect{v}}\sum_{z\in X\setminus \sect{v}}|K(z,x)-K(z,y)|q^{-\alpha|z|}<+\infty. 
	\end{equation}
	If $K$ satisfies~\eqref{hormintro}, an integral operator $\cK$ with kernel $K$ that is bounded on $L^2(\mu_\alpha)$ it is also bounded on $H^1$ and, by interpolation, on $L^p(\mu_\alpha)$ for $1<p<2$. By duality, if $K^*(x,y)=\overline{K(y,x)}$ satisfies~\eqref{hormintro}, then $\cK$ is bounded on ${\rm BMO}$ and on $L^p(\mu_\alpha)$, for $2<p<\infty$.
	
	A natural question is whether there are other measures on $X$ for which this approach can be replicated. We try to answer in the final section. We focus our attention on the class of reference measures introduced in~\cite{ccps} in the definition of harmonic Bergman spaces on $X$. We provide a characterization of the subfamily of the reference measures that are doubling with respect to the Gromov distance, for which the main results can be obtained.
	
	\section{Preliminaries}
	Let $X$ be a $q$-homogeneous tree with $q>1$, that is a connected and loop-free graph in which every vertex is joined with exactly $q+1$ vertices. The tree is endowed with the canonical discrete distance $d$ defined by the number of edges lying in the unique finite path joining the two vertices. We fix an origin $o\in X$ and we set $|x|=d(o,x)$ for every $x\in X$. We denote the sphere and the ball of radius $n\in\N$ centered in $x\in X$ respectively by
	\[\sph{x}{n}=\{y\in X\colon d(x,y)=n \},\qquad \ball{x}{n}=\{y\in X\colon d(x,y)\leq n \}. \]
	We call \textit{predecessor} of $x\in X\setminus\{o\}$ the unique neighbor $p(x)$ of $x$ such that $|p(x)|=|x|-1$. It is useful to consider the predecessor as a (surjective but not injective) function $p\colon X\setminus\{o\}\to X$ so that its $\ell$-power is $p^\ell\colon X\setminus\ball{o}{\ell-1}\to X$. Furthermore we call
	 \textit{successors} of $x\in X$ all neighbors of $x$ different from the predecessor, and we denote the family of successors by $s(x)$. The \textit{sector} of $x\in X$ is
	 \[\sect{x}:=\{y\in X\colon x=p^\ell(y), \text{ for some }\ell\in\N\}\subseteq X. \]
	 Observe that $p^{|y|}(y)=o$ for every $y\in X$ and then $\sect{o}=X$. Given $x,y\in X$, we call the \textit{confluent} of $x$ and $y$ the furthest vertex from the origin $x\wedge y\in X$ satisfying
	 $x\wedge y=p^{|x|-|x\wedge y|}(x)=p^{|y|-|x\wedge y|}(y)$, or, equivalently, $\{x,y\}\subseteq\sect{x\wedge y}$. Clearly, $x\wedge x=x$ and $x\wedge o=o$.
	
	We introduce another distance on $X$, usually called \textit{Gromov distance}, see~\cite{ArcRoc} and~\cite{Gromov}, defined as
	\begin{equation*}
		\rho(x,y):=\begin{cases}
			0,&\text{if }x=y,
			\\ e^{-|x\wedge y|}, &\text{otherwise.}
		\end{cases}
	\end{equation*}
	For every $x\in X\setminus\{o\}$, if $y\in X\setminus\{x\}$ then $\rho(x,y)=e^{-|x\wedge y|}\in [e^{-|x|},1]$ and $|x\wedge y|=-\log(\rho(x,y))$, so that we have \[y\in\sect{p^{|x|+\log(\rho(x,y))}(x)}\setminus\sect{p^{|x|+\log(\rho(x,y))-1}(x)}.\]
	Hence, the nontrivial balls with respect to $\rho$ having center in $x$ are sectors of the tree. More in general, we have
	\begin{equation}\label{balle}
		B_\rho(x,r):=\{ y\in X\colon\rho(x,y)<r\}=\begin{cases}
			\{x\},&\text{if }0<r\leq e^{-|x|},\\
			T_{p^{|x|+\lfloor\log r\rfloor}(x)},&\text{if }  e^{-|x|}<r\leq 1,\\
			X,&\text{if }r>1.
		\end{cases}
	\end{equation}
	Observe that in the special case $x=o$ we have that $B_\rho(o,r)=\{o\}$ if $0<r\leq 1$ and $B_\rho(o,r)=X$ for every $r>1$.
	Hence, every vertex $x$ is the center of exactly $|x|+2$ balls.

	The aim of this work is to study the boundedness of integral operators on homogeneous trees with respect to a certain class of measures. 
	We consider the family of \textit{exponentially decreasing radial measures} defined, for every $\alpha>1$, by
	\[\mu_\alpha(x):=q^{-\alpha|x|},\qquad x\in X.\]
	We set $L_\alpha^p:=L^p(\mu_\alpha)$ and $\|\cdot\|_{p,\alpha}:=\|\cdot\|_{L_\alpha^p}$.
	It is easy to check that such measures are finite on $X$. Furthermore, although they are not doubling with respect to the distance $d$ (see~\cite{dmv} for detail), they are doubling with respect to the distance $\rho$.
	
	\begin{proposition}[\cite{dmv}]\phantom{\label{doubling}}
		The triple $(X,\rho,\mu_\alpha)$ is globally doubling for every $\alpha>1$ with doubling constant $C_\alpha=\max\{q^\alpha+1,(1-q^{1-\alpha})^{-1}\}$, that is 
		\begin{equation}\label{eqdoubling}
			\mu_\alpha(B_\rho(x,2r))\leq C_\alpha\mu_\alpha(B_\rho(x,r)),\qquad x\in X,\,r>0.
		\end{equation}
	\end{proposition}
	\begin{proof}
		Let $\alpha>1$. We start by computing for every $x\in X\setminus\{o\}$
		\begin{equation}\label{misurasettore}
			\mu_\alpha(\sect{x})=\sum_{\ell=0}^{+\infty}q^{\ell}q^{-\alpha(\ell+|x|)}=q^{-\alpha|x|}\frac{1}{1-q^{1-\alpha}}.
		\end{equation}
		Clearly, it is sufficient to prove~\eqref{eqdoubling} for $r\in(0,1]$, since $B_\rho(x,r)=B_\rho(x,2r)=X$ for every $r> 1$.
		Let $0<r\leq 1$. We put $\{z\}:=z-\lfloor z\rfloor\in[0,1)$ and we have
		\begin{equation*}
			\lfloor \log(2r)\rfloor=\begin{cases}
				\lfloor \log r\rfloor,&\text{if }0\leq \{\log r\}<1-\log 2,\\
				1+\lfloor \log r\rfloor,&\text{if }1-\log 2\leq \{\log r\}<1.
			\end{cases}
		\end{equation*}
		Hence, whenever $B_\rho(x,r)=\{x\}$ we have that $B_\rho(x,2r)\in \{\{x\},\sect{x}\}$, and if $B_\rho(x,r)=\sect{x}$ then $B_\rho(x,2r)\in\{\sect{x},\sect{p(x)}\}$.
		Now we show that the measures of the balls are uniformly comparable. If $x\in X\setminus\{o\}$, then by~\eqref{misurasettore}
	\begin{equation}\label{rapportosettvert}
			\frac{\mu_\alpha(\sect{x})}{\mu_\alpha(\{x\})}=\frac{q^{-\alpha|x|}(1-q^{1-\alpha})^{-1}}{q^{-\alpha|x|}}=\frac{1}{1-q^{1-\alpha}}.
	\end{equation}
		If $|x|>1$, then 
		\begin{equation}\label{rapportosettsettpadre}
			\frac{\mu_\alpha(\sect{p(x)})}{\mu_\alpha(\sect{x})}=\frac{q^{-\alpha(|x|-1)}(1-q^{1-\alpha})^{-1}}{q^{-\alpha|x|}(1-q^{1-\alpha})^{-1}}=q^\alpha.
		\end{equation}
		Finally, if $|x|=1$, then 
		\begin{equation}\label{rapportosetttutto}
			\frac{\mu_\alpha(X)}{\mu_\alpha(\sect{x})}=\frac{(1+q^{-\alpha})(1-q^{1-\alpha})^{-1}}{q^{-\alpha}(1-q^{1-\alpha})^{-1}}=q^\alpha+1.
		\end{equation}
		Hence $(X,\rho,\mu_\alpha)$ is doubling with constant $C_\alpha=\max\{q^\alpha+1,(1-q^{1-\alpha})^{-1}\}$.
	\end{proof}
	
%
%

	The family of exponential decreasing reference measures can be view as the natural counterpart of the measures $(1-|x+iy|^2)^{\alpha-2}$ on the hyperbolic disk. For the functions that are integrable with respect to such measures, we prove a Calder\'on-Zygmund decomposition; then we introduce $H^1$ and BMO spaces and we discuss interpolation properties. In the next section, we prove a boundedness result for integral operators.
	
	We start with a preliminary geometrical result that shows the existence of a family $\mathcal{D}$ of subsets of $X$ formed by an infinite family of partitions of $X$ in singletons and sectors. In particular, the partition at a given scale is a refinement of the partition at the previous scale, and the measure of a partitioning set is comparable with the measure of the set which contains it in the previous partition. The family $\cD$ can be thought of as the analogous of the family of dyadic sets in the Euclidean case.
	Observe that $\mathcal{D}$ does not depend on $\alpha>1$.
	\begin{lemma}[Lemma~29~\cite{dmv}\label{lemdecomp}]
		For every $m\in\N$, there exists $I_m\in\N$ and sets $D_{k,m}\subseteq X$ for every $k\in\cI_m:=\{0,\dots,I_m\}$ such that the family $\cD$ defined by
		\[\cD:=\{D_{m,k}\subseteq X\colon m\in\N,\, k\in\cI_m\},\]
		satisfies:
		\begin{enumerate}
			\item\label{i} for every $m\in\N$, the family $\cD_m:=\{D_{m,k}\colon k\in\cI_m\}$ is a partition of $X$;
			\item\label{ii} the partition $\cD_m$ at scale $m>0$ is a refinement of the partition $\cD_{m-1}$, that is, for every $k'\in\cI_{m-1}$ there exists $\cI_{m,k'}\subseteq\cI_m$ such that \[D_{k',m-1}=\bigsqcup_{k\in\cI_{m,k'}}D_{k,m};\]
			\item\label{iii} for every $k\in\cI_m$ and $k'\in\cI_{m-1}$ for which $D_{k,m}\subseteq D_{k',m-1}$, we have
			\[\mu_\alpha(D_{k,m})\leq\mu_\alpha(D_{k',m-1})\leq C_{\alpha}\mu_\alpha(D_{k,m});\]
			\item\label{iiii} for every $v\in X$, $\{v\}\subseteq \cD_m$, whenever $m\geq |v|$.
		\end{enumerate}
	\end{lemma}
	\begin{proof}
		For every $m\in\N$ we set
		\[I_m:=\# \ball{o}{m}-1=\begin{dcases*}
			0,&if $m=0$;\\
			\frac{q^{m+1}+q^{m}-q-1}{q-1},&if $m>0$.
		\end{dcases*}.\]
		We label the vertices in such a way that $v_0=o$, $s(o)=\sph{o}{1}=\{v_1,\dots,v_{q+1}\}$, and $s(v_k)=\{v_{qk+\ell}\colon \ell\in\{1,\dots,q\}\}$ for every $k\in\N\setminus\{0\}$. Since $\cI_0=\{0\}$, it is sufficient to set $D_{0,0}=X$. Then, for every $m\in \N\setminus\{0\}$, we set
		\begin{alignat*}{2}
			D_{k,m}&:=\{v_k\},\qquad&&\text{if }k\in\cI_{m-1},\iff \text{ if }  v_k\in\ball{o}{m-1},\\ 
			D_{k,m}&:=\sect{v_k},&&\text{if }k\in\cI_m\setminus\cI_{m-1}\iff\text{ if } v_k\in\sph{o}{m}. 
		\end{alignat*}
		In this way, (i), (ii), and (iv) easily follow by construction.   Finally,~(iii) follows from~\eqref{rapportosettvert},~\eqref{rapportosettsettpadre},~\eqref{rapportosetttutto}, and the fact that for $m>0$
		\begin{equation}
			D_{k',m-1}\in\begin{dcases*}
				\{\sect{v},\sect{p(v)} \},&if $D_{k,m}=\sect{v}$;\\
				\{\{v\},\sect{v} \},&if $D_{k,m}=\{v\}$.
			\end{dcases*}
		\end{equation}
	\end{proof}
	
	We define the \textit{Hardy-Littlewood maximal function} $M$ with respect to $\mu_\alpha$ associated to the family $\cD$ as follows
	\[Mf(x)=\sup_{\substack{x\in D\\D\in\cD}}\frac{1}{\mu_\alpha(D)}\sum_{z\in D}|f(z)|q^{-\alpha|z|},\qquad f\colon X\to\C.\]
	As a consequence of the decomposition presented in Lemma~\ref{lemdecomp}, we obtain the following result.
	\begin{proposition}\label{hlweak}
		The Hardy-Littlewood maximal function $M$ is of weak type (1,1) and bounded on $L_\alpha^p$, for every $1<p\leq\infty$.
	\end{proposition}
	\begin{proof}
		The boundedness of $M$ on $L_\alpha^\infty$ easily follows from
		\[|Mf(x)|=\sup_{\substack{x\in D\\D\in\cD}}\frac{1}{\mu_\alpha(D)}\sum_{z\in D}|f(z)|q^{-\alpha|z|}\leq \|f\|_{\infty,\alpha},\qquad x\in X. \]
		Now we prove that $M$ is of weak type (1,1). Let $\lambda>0$. We set
		\[\Omega_\lambda:=\{x\in X\colon Mf(x)>\lambda\}.\]
		 If $\lambda\leq(\mu_\alpha(X))^{-1}\|f\|_{1,\alpha}$, then
		 \[\mu_\alpha(\Omega_\lambda)\leq\mu_\alpha(X)\leq\frac{\|f\|_{1,\alpha}}{\lambda}. \]
		 Consider the case $\lambda>(\mu_\alpha(X))^{-1}\|f\|_{1,\alpha}$. There exists a partition $\{E_i\in\cD\colon i\in I\}$ which is at most countable such that 
		 \begin{equation}\label{medialambda}
		 	\frac{1}{\mu_\alpha(E_i)}\sum_{y\in E_i}|f(y)|q^{-\alpha|y|}>\lambda.
		 \end{equation} 
	 	Indeed such partition is obtained by considering the maximal dyadic sets satisfying~\eqref{medialambda} and using that fact that if $E_1,E_2\in\cD$, then, by their definition, either $E_1\cap E_2=\emptyset$ or $E_1\subseteq E_2$ (up to switch the two sets).
		 Hence, we have
		 \[\mu_\alpha(\Omega_\lambda)\leq \sum_{i\in I}\mu_\alpha(E_i)\leq\frac{1}{\lambda}\sum_{i\in I}\|f\|_{L_\alpha^1(E_i)}\leq\frac{\|f\|_{1,\alpha}}{\lambda}. \]
		 The boundedness of $M$ on $L_\alpha^p$ for $1<p<\infty$ follows from interpolation.
	\end{proof}
	Lemma~\ref{lemdecomp} leads to a Calder\'on-Zygmund decomposition for integrable functions on $X$ at level $\lambda\in\R^+$, sufficiently large with respect to the $L_\alpha^1$-norm of the function.
	\begin{proposition}[Proposition~30~\cite{dmv}\label{caldzyg}]
		Let 
		$f\in L_\alpha^1$ and $\lambda>\|f\|_{1,\alpha}/\mu_\alpha(X)$. There exist two families $\cQ$ and $\cF$ of disjoint sets in $\cD$ such that, if we denote by $\Omega$ and $F$ the disjoint union of all the sets in $\cQ$ and $\cF$, respectively, 
		the following properties hold:
		\begin{enumerate}
			\item $X=\Omega\sqcup F$;
			\item $|f(z)|\leq \lambda$ for every $z\in F$;
			\item there exist $g,\,b\colon X\to\C$ and $C>0$ such that $f=g+b$, $\supp\, b\subseteq \Omega$, and $\|g\|_{2,\alpha}^2\lesssim \lambda\|f\|_{1,\alpha}$. Moreover, if we set $b_Q=b\mathbbm{1}_Q$ for every $Q\in\cQ$, then
			\[\sum_{z\in Q} b_Q(z)q^{-\alpha|z|}=0,\qquad \sum_{Q\in\cQ}\|b_Q\|_{1,\alpha}\leq C\|f\|_{1,\alpha},\qquad Q\in\cQ.\]
			%
			%
		\end{enumerate}
	\end{proposition}
	\begin{proof}
		 We define two families $\mathcal{Q}$ and $\cF$ of subsets of the decomposition $\cD$ of the tree presented in Lemma~\ref{lemdecomp}. Starting from $D_{0,0}=X$:
		\begin{enumerate}
			\item[1)] if 
			\[	\frac 1{\mu_\alpha(D_{k,m})}\sum_{z\in D_{k,m}}|f(z)|q^{-\alpha|z|}> \lambda,\]
			then we put $D_{k,m}\in\mathcal{Q}$ and we stop. Otherwise,
			\item[2a)] if $\# D_{k,m}=1$ then $D_{k,m}\in \cF$ and we stop;
			\item[2b)] if $\# D_{k,m}>1$ then for each set in the family \[D_{k,m+1}\cup\{D_{kq+j,m+1}\colon j\in{1,\dots q}\}\]
			we repeat the procedure, starting from 1).
		\end{enumerate}
		Observe that $X\not\in\cQ$ because $\lambda>(\mu_\alpha(X))^{-1}\|f\|_{1,\alpha}$.
		We denote by $\Omega$ and $F$ the (disjoint) union of all the subsets in $\cQ$ and $\cF$, respectively. The sets $\Omega$ and $F$ clearly satisfy (i) and (ii).
		We prove that, for every $Q\in\cQ$,
		\begin{equation}\label{tipot}
			\lambda<  \frac 1{\mu_\alpha(Q)}\sum_{z\in Q}|f(z)|q^{-\alpha|z|}\leq C_\alpha \lambda,\qquad Q\in\cQ.
		\end{equation}
		For every $Q=D_{k,m}\in \cQ$, we have $m>0$ since $X\not\in\cQ$ and we put $\tilde{Q}=D_{k',m-1}$,
		where $k'$ is defined in (iv) of Lemma~\ref{lemdecomp}. Observe that $\tilde{Q}\not\in\cQ$ and that, by Lemma~\ref{lemdecomp}, $\mu_\alpha(\tilde{Q})\leq C_\alpha\mu_\alpha(Q)$. Then we have that
		\[\frac{1}{\mu_\alpha(Q)}\sum_{z\in Q}|f(z)|q^{-\alpha|z|}\leq \frac{\mu_\alpha(\tilde Q)}{\mu_\alpha(Q)}\frac{1}{\mu_\alpha(\tilde Q)}\sum_{z\in \tilde Q}|f(z)|q^{-\alpha|z|}\leq C_\alpha \lambda,\]
		which gives~\eqref{tipot}.
		It is easy to see that 
		\begin{equation}\label{Omegaminore}
			\mu_\alpha(\Omega)\leq \frac{1}{\lambda}\sum_{Q\in\cQ}\sum_{x\in Q}\frac{1}{\mu_\alpha(Q)}\left(\sum_{z\in Q}|f(z)|q^{-\alpha|z|}\right)q^{-\alpha|x|}\leq \frac{\|f\|_{1,\alpha}}{\lambda}.
		\end{equation}
		
		We now define $b=f-g$, where
		\begin{equation*}
			g(z)=\begin{dcases*}
				f(z),&$z\in F$,\\
				\frac1{\mu_\alpha(Q)}\sum_{x\in Q}f(x)q^{-\alpha|x|},&$z\in Q,\,Q\in\cQ$.
			\end{dcases*}
		\end{equation*}
		It is obvious that $\supp\,b\subseteq\Omega$. We show next that $\|g\|_{2,\alpha}^2\leq (1+C_\alpha^2)\lambda\|f\|_{1,\alpha}$. Indeed, by~\eqref{tipot},
		\begin{align*}
			\|g\|_{2,{\alpha}}^2&=\sum_{z\in F}|g(z)|^2q^{-\alpha|z|}+\sum_{z\in \Omega}|g(z)|^2q^{-\alpha|z|}\\
			&=\sum_{z\in F}|f(z)|^2q^{-\alpha|z|}+\sum_{Q\in\cQ}\sum_{z\in Q}\left|\frac1{\mu_\alpha(Q)}\sum_{x\in Q}f(x)q^{-\alpha|x|}\right|^2q^{-\alpha|z|}\\
			&\leq\sum_{z\in F}\lambda|f(z)|q^{-\alpha|z|}+\mu_\alpha(\Omega) C_\alpha^2\lambda^2\leq
			(1+C_\alpha^2)\lambda\|f\|_{1,\alpha}<+\infty,
		\end{align*}
		where we used~\eqref{Omegaminore}.
		The fact that $b_Q=b\mathbbm{1}_Q$, $Q\in\cQ$, has vanishing mean on $Q$ follows by construction. Furthermore, since $|b(z)|\leq |f(z)|+|g(z)|$ we have
		\begin{align*}
			\sum_{Q\in \cQ}\sum_{z\in Q}|b_Q(z)|q^{-\alpha|z|}&\leq \sum_{z\in \Omega}|f(z)|q^{-\alpha|z|}+ \sum_{Q\in \cQ}\sum_{z\in Q}|g(z)|q^{-\alpha|z|}\\
			&\leq \|f\|_{1,\alpha}+\mu_\alpha(\Omega)C_\alpha \lambda\lesssim \|f\|_{1,\alpha},
		\end{align*}
		by~\eqref{Omegaminore}.
	\end{proof}

	\section{$H^1$ and BMO spaces}
	In this section we define $H^1$ and BMO spaces associated to $\mu_\alpha$, $\alpha>1$. A classical reference  for the theory on these classes of spaces is~\cite{coifweiss} to which we refer for the proof of some of the classical results. See~\cite{arditti1},~\cite{arditti2},~\cite{feneuil},~\cite{kpt}, and~\cite{russ} for a theory on discrete sets.
	
	We start by defining $(1,p)$-atoms and consequently the atomic Hardy space $H^{1,p}_\alpha$.
	\begin{definition}
	Let $1< p\leq \infty$. A function $a$ is a $(1,p)$-atom if either  $a=\mu_\alpha(X)^{-1}$ or 
	\begin{enumerate}
		\item $a$ is supported in $D$ for some $D\in\cD$;
		\item $\|a\|_{p,\alpha}\leq\mu_\alpha(D)^{\frac{1}{p}-1}$ ($\|a\|_{\infty,\alpha}\leq\mu_\alpha(D)^{-1}$, if $p=\infty$);
		\item the mean of $a$ on $D$ vanishes, that is
		\[\sum_{x\in D}a(x)q^{-\alpha|x|}=0.\]
	\end{enumerate}
	\end{definition}
	\begin{definition}
		We define the space $H^{1,p}_\alpha$ as the space of the functions $g\in L^1_\alpha$ such that 
		\[ g=\sum_{j}\lambda_ja_j,\]
		where $a_j$ are $(1,p)$-atoms and $\lambda_j\in\C$ is a summable sequence. We set
		\[\|g\|_{H^{1,p}_\alpha}:=\inf\left\{\sum_j|\lambda_j|\colon g=\sum_j\lambda_ja_j ,\,\,a_j\,\,(1,p)-\text{atoms}\right\}. \] 
		
	\end{definition}

	For every function $f$ on the tree and $D\in\cD$, we denote by $f_D$ the average of $f$ on $D$, that is 
	\[f_{D}=\frac{1}{\mu_\alpha(D)}\sum_{x\in D}f(x)q^{-\alpha|x|}.\]
	\begin{definition}
		Let $1\leq r<\infty$. We define ${\rm BMO}_{r,\alpha}$ as the space of all the functions $f\colon X\to\C$ such that 
		\[\|f\|_{{\rm BMO}_{r,\alpha}}:=\sup_{D\in \cD}\left(\frac{1}{\mu_\alpha(D)}\sum_{x\in D}|f(x)-f_D|^rq^{-\alpha|x|}\right)^\frac{1}{r} +\left|\sum_{x\in X}f(x)q^{-\alpha|x|}\right|\]
		quotiented over the constant functions.
	\end{definition}
	It is easy to check that $({\rm BMO}_{r,\alpha},\|\cdot\|_{{\rm BMO}_{r,\alpha}})$ is a Banach space. Furthermore, BMO spaces are in general inboxed, namely ${\rm BMO}_{r,\alpha}\subseteq{\rm BMO}_{1,\alpha}$ for every $1\leq r<\infty$. Indeed by H\"older inequality, if $E=\sect{v}$,
	\[\sum_{x\in\sect{v}}|f(x)-f_{\sect v}|q^{-\alpha|x|}\leq \mu_\alpha(\sect{v})^ {1-\frac{1}{r}} \sum_{x\in\sect{v}}|f(x)-f_{\sect v}|^rq^{-\alpha|x|},\]
	and hence $\|f\|_{{\rm BMO}_{1,\alpha}}\leq\|f\|_{{\rm BMO}_{r,\alpha}}$.
	
	By the fact that $(X,\rho,\mu_\alpha)$ is doubling we can apply Theorem A of~\cite{coifweiss}. Hence we have that $H_\alpha^{1,p}=H_\alpha^{1,\infty}$ as vector spaces and their norms are equivalent for every $1<p<\infty$.
	
	\begin{proposition}[Theorem B~\cite{coifweiss}\label{propduale}]
			Let $1<p\leq \infty$. For every $\Phi\in(H_\alpha^{1,p})^*$ there exists a function $f\in{\rm BMO}_{p',\alpha}$ such that $\|\Phi\|_{(H_{\alpha}^{1,p})^*}\asymp\|f\|_{{\rm BMO}_{p',\alpha}}$ and for every $(1,p)$-atom $a$ we have
			\begin{equation}\label{dualoper}
				\Phi(a)=\sum_{x\in X}f(x)a(x)q^{-\alpha|x|}.
			\end{equation}
			On the other hand, for every $f\in{\rm BMO}_{p',\alpha}$, the functional defined as in~\eqref{dualoper} on $(1,p)$-atoms extends to a unique linear bounded functional on $H_\alpha^{1,p}$ whose norm is equivalent to $\|f\|_{{\rm BMO}_{p',\alpha}}$.
	\end{proposition}

	As consequence of the equivalence of $H_\alpha^{1,p}$ spaces and Proposition~\ref{propduale}, we have that ${\rm BMO}_{r,\alpha}={\rm BMO}_{1,\alpha}$ as vector spaces with equivalent norms for every $1<r<\infty$.
	In the sequel, we shall denote by $H_\alpha^1$ the space $H_\alpha^{1,\infty}$ and by ${\rm BMO}_{\alpha}$ the space ${\rm BMO}_{1,\alpha}$.

	\section{Complex interpolation}\label{secinter}
	In this section we show that the classical results for the complex interpolation involving $H^1$ and BMO spaces in the Euclidean setting hold also in this setting (see, for example,~\cite{cmm}, ~\cite{lstv} and~\cite{vallax+b}).
	
	We  start by defining the sharp maximal function $M^\sharp$ of $f\colon X\to \C$ as 
	\[M^\sharp f(x):=\sup_{\substack{D\in\cD\\x\in D}}\frac{1}{\mu_\alpha(D)}\sum_{y\in D}|f(y)-f_D|q^{-\alpha|y|},\qquad x\in X.\]
	The next technical result is usually known as \lq\lq good lambda inequality\rq\rq\,and it is crucial in the proof of complex interpolation results.
	\begin{proposition}\label{propgoodlambda}
		There exists $C'>0$ such that for every $\gamma>0,\,\lambda>0$ and $f\colon X\to\C$ we have
		\begin{equation}\label{goodlambda}
			\mu_\alpha(\{x\in X\colon Mf(x)>2\lambda,\,M^\sharp f(x)<\gamma\lambda\})\leq C'\gamma\mu_\alpha(\{x\in X\colon Mf(x)>\lambda\}). 
		\end{equation}
		\end{proposition}
	\begin{proof}
		Let $\Omega_\lambda=\{x\in X\colon Mf(x)>\lambda\}$. For every $x\in \Omega_\lambda$ there exists a $D_x\in\cD$ that is maximal with respect to the inclusion such that
		\[\frac{1}{\mu_\alpha(D_x)}\sum_{y\in D_x}|f(y)|q^{-\alpha|y|}>\lambda. \]
		Since $D_y=D_x$ for every $y\in D_x$, it is sufficient to prove~\eqref{goodlambda} restricted to $D_x$, namely
		\begin{align}\label{goodlambda2}
			\mu_\alpha(\{x\in D_x\colon Mf(x)>2\lambda,\,M^\sharp f(x)<\gamma\lambda\})&\leq C' \gamma \mu_\alpha(\{x\in D_x\colon Mf(x)>\lambda\})\nonumber\\
			&=C' \gamma \mu_\alpha(D_x).
		\end{align}
	Fix $x\in \Omega_\lambda$ and $y\in D_x$ such that $Mf(y)>2\lambda$, then the supremum
	\[\sup_{\substack{D\in\cD\\y\in D}}\frac{1}{\mu_\alpha(D)}\sum_{z\in D}|f(z)|q^{-\alpha|z|}= \sup_{\substack{D\in\cD\\D\subseteq D_x\wedge D_x\subseteq D}}\frac{1}{\mu_\alpha(D)}\sum_{z\in D}|f(z)|q^{-\alpha|z|}.\]
	If $D_x\subseteq D$, then by definition of $D_x$
	\[\frac{1}{\mu_\alpha(D)}\sum_{z\in D}|f(z)|q^{-\alpha|z|}\leq \lambda,\]
	and then $Mf(x)=M(f\mathbbm{1}_{D_{x}})(x)$. Let $D_x'\in\cD$ be the smaller dyadic set such that $D_x\subsetneq D_x'$ (see Lemma~\ref{lemdecomp}), then
	\begin{align*}
		M((f-f_{D_x'})\mathbbm{1}_{D_x})(x)&=M(f\mathbbm{1}_{D_{x}})(x)-|f_{D_x'}|\\
	&\geq Mf(x)-\frac{1}{\mu_\alpha(D_x')}\sum_{y\in D_x'}|f(y)|q^{-\alpha|y|}>2\lambda-\lambda=\lambda. 
	\end{align*}
	Hence
	\[\mu_\alpha(\{y\in D_x\colon Mf(y)>2\lambda\})\leq \mu_\alpha(\{y\in D_x\colon M((f-f_{D_x'})\mathbbm{1}_{D_x})(y)>\lambda\}). \]
	By Proposition~\ref{hlweak}, $M$ is of weak type (1,1) and, if we denote by $\|M\|$ its operatorial norm as operator from $L_\alpha^1$ to the Lorentz space $L_\alpha^{1,\infty}$, then
	\begin{align*}
		\mu_\alpha(\{y\in D_x\colon Mf(y)>2\lambda\})&\leq\frac{\|M\|}{\lambda}\|(f-f_{D_x'})\mathbbm{1}_{D_x}\|_{1,\alpha} \\
		& =\frac{\|M\|}{\lambda}\sum_{y\in D_x}|f(y)-f_{D_x'}|q^{-\alpha|y|} \\
		& \leq\frac{\|M\|}{\lambda}C_\alpha\frac{\mu_\alpha(D_x)}{\mu_\alpha(D_x')}\sum_{y\in D_x'}|f(y)-f_{D_x'}|q^{-\alpha|y|} \\
		& \leq\frac{\|M\|}{\lambda}C_\alpha\mu_\alpha(D_x)M^\sharp f(\xi_x),
	\end{align*}
	for some $\xi_x\in D_x$. If $\xi_x\in D_x$ is such that $M^\sharp f(\xi_x)\leq \gamma\lambda$, then
	\begin{align*}
		\mu_\alpha(\{y\in D_x\colon Mf(y)>2\lambda,\,M^\sharp f(y)\leq\gamma\lambda\})&\leq \frac{\|M\|}{\lambda}C_\alpha\mu_\alpha(D_x)M^\sharp f(\xi_x)\\
		&\leq \gamma\|M\|C_\alpha\mu_\alpha(D_x),
	\end{align*}
	that is~\eqref{goodlambda2}, as required.
	\end{proof}
	As consequence of Proposition~\ref{propgoodlambda} we have the following result.
	\begin{theorem}
		Let $1\leq p_0<\infty$. For every $p_0\leq p<\infty$, there exists $N_p>0$ such that for every $f\colon X\to \C$ with $Mf\in L^{p_0}_\alpha$ we have
		\begin{enumerate}
			\item $\|Mf\|_{p,\alpha}\leq N_p\|M^\sharp f\|_{p,\alpha}$;
			\item $\|f\|_{p,\alpha}\leq N_p\|M^\sharp f\|_{p,\alpha}$.
		\end{enumerate}
	\end{theorem}
	\begin{proof}
		For every $N>0$ we set
		\[\cJ_N:=\int_0^Np\lambda^{p-1}\mu_\alpha(\{x\in X\colon Mf(x)>\lambda \})\de\lambda.\]
		Observe that for every $N$, $\cJ_N$ is finite, indeed
		\begin{align*}
			\cJ_N&= \int_0^Np\lambda^{p_0-1}\frac{\lambda^{p-1}}{\lambda^{p_0-1}}\mu_\alpha(\{x\in X\colon Mf(x)>\lambda \})\de\lambda\\
			&\leq N^{p-p_0}\frac{p}{p_0}\int_0^Np_0\lambda^{p_0-1}\mu_\alpha(\{x\in X\colon Mf(x)>\lambda \})			\de\lambda\\
			&\leq N^{p-p_0}\frac{p}{p_0}\|Mf\|_{p_0,\alpha}^{p_0}<+\infty.
		\end{align*}
		Now observe that, by changing the variable $\lambda\to\frac{\lambda}{2}$ and using Proposition~\ref{propgoodlambda},
		\begin{align*}
			\cJ_N=2^p&\int_{0}^{\frac N2}p\lambda^{p-1}\mu_\alpha(\{x\in X\colon Mf(x)>2\lambda \})\de\lambda\\
			=2^p&\int_{0}^{\frac N2}p\lambda^{p-1}\mu_\alpha(\{x\in X\colon Mf(x)>2\lambda,\,M^\sharp f(x)\leq\gamma\lambda\})\de\lambda\\
			&+2^p\int_{0}^{\frac N2}p\lambda^{p-1}\mu_\alpha(\{x\in X\colon Mf(x)>2\lambda,\,M^\sharp f(x)>\gamma\lambda\})\de\lambda\\
			\leq 2^p&\gamma\|M\|C_\alpha\int_{0}^{\frac N2}p\lambda^{p-1}\mu_\alpha(\{x\in X\colon Mf(x)>\lambda\})\de\lambda\\
			&+2^p\int_{0}^{\frac N2}p\lambda^{p-1}\mu_\alpha(\{x\in X\colon Mf(x)>2\lambda,\,M^\sharp f(x)>\gamma\lambda\})\de\lambda\\
			\leq 2^p&\gamma\|M\|C_\alpha\cJ_N+\frac{2^p}{\gamma^p}\int_{0}^{\frac {N\gamma}2}p\lambda^{p-1}\mu_\alpha(\{x\in X\colon M^\sharp f(x)>\lambda\})\de\lambda,
		\end{align*}
	by applying the change $\lambda\to\gamma\lambda$.
	Now, by choosing $\gamma^{-1}=2^{p+1}\|M\|C_\alpha$, we have that 
	\[\frac 12\cJ_N\leq \frac{2^p}{\gamma^p}\int_{0}^{\frac {N\gamma}2}p\lambda^{p-1}\mu_\alpha(\{x\in X\colon M^\sharp f(x)>\lambda\})\de\lambda,\]
	that is
	\[\cJ_N\leq \frac{2^{p+1}}{\gamma^p}\int_{0}^{\frac {N\gamma}2}p\lambda^{p-1}\mu_\alpha(\{x\in X\colon M^\sharp f(x)>\lambda\})\de\lambda.\]
	It is sufficient to notice that when $N\to+\infty$ the LHS is $\|Mf\|_{\alpha,p}^p$ and the RHS is $\|M^\sharp f\|_{\alpha,p}^p$. Hence (i) follows with $N_p=2^\frac{p+1}{p}\gamma^{-1}$.
		
		Finally, (ii) follows from $|f|\leq Mf$.
	\end{proof}

	We are now in a position to study complex interpolation spaces involving $H^1_\alpha$ and ${\rm BMO}_\alpha$. To do so, we shall recall some basic notions of complex interpolation. We refer the reader to~\cite{berghlof} for more details.

	Let $A_0,\,A_1$ be a pair of normed spaces and $\theta\in (0,1)$. We denote by $\Sigma\subseteq\C$ the complex strip formed by $z\in\C$ with $\Re(z)\in(0,1)$ and by $\overline{\Sigma}$ its closure. We define the space $\cF(A_0,A_1)$ of the functions $F\colon\overline{\Sigma}\to A_0+A_1$ such that
	\begin{enumerate}
		\item for every $\ell\in(A_0+A_1)^*$, the function $z\to\langle F(z),\ell\rangle_{(A_0+A_1)\times(A_0+A_1)^*}$ is continuous and bounded on $\overline{\Sigma}$ and analytic on $\Sigma$;
		\item $F(it)$ is bounded on $A_0$ for every $t\in\R$;
		\item $F(1+it)$ is bounded on $A_1$ for every $t\in\R$.
	\end{enumerate}
	The space $\cF(A_0,A_1)$ is equipped with the norm
	\[\|F\|_{\cF}=\sup_{t\in\R}\{\max(\|F(it)\|_{A_0},\|F(1+it)\|_{A_1}) \}<+\infty,\qquad F\in\cF. \]
	We define the (complex) interpolation space between $A_0$ and $A_1$ as
	\[(A_0,A_1)_{[\theta]}=\{f\in A_0+A_1\colon f=F(\theta),\text{ for some }F\in\cF(A_0,A_1) \},\]
	that is a normed space with the norm defined on $f\in (A_0,A_1)_{[\theta]}$ as
	\[\|f\|_{(A_0,A_1)_{[\theta]}}=\|f\|_{[\theta]}:=\inf\{\|F\|_{\cF}\colon F\in\cF(A_0,A_1),\,F(\theta)=f \}<+\infty. \]

	The following results represent the heart of the study of the complex interpolation between $H_\alpha^1$, ${\rm BMO}_\alpha$ and $L_\alpha^p$ spaces.
	\begin{theorem}\label{thmintcompl}
		The following interpolation results hold:
		\begin{enumerate}
			\item if $\theta\in(0,1)$ and $1<p_1<p<\infty$ such that $\frac{1}{p}=\frac{1-\theta}{p_1}$, then 
			\[(L_\alpha^{p_1},{\rm BMO}_\alpha)_{[\theta]}=L_\alpha^p;\]
			\item if $\theta\in(0,1)$ and $1<p<p_1<\infty$ such that $\frac{1}{p}=1-\theta+\frac{\theta}{p_1}$, then 
			\[(H_\alpha^1,L_\alpha^{p_1})_{[\theta]}=L_\alpha^p.\]
		\end{enumerate}
	\end{theorem}
	\begin{proof}
		First we observe that $L_\alpha^p\subseteq(L_\alpha^{q_1},{\rm BMO}_\alpha)_{[\theta]}$ follows from the continuous embedding of $L_\alpha^\infty$ in ${\rm BMO}_\alpha$ and from  $L_\alpha^p=(L_\alpha^{p_1},L_\alpha^\infty)_{[\theta]}$, see Theorem~5.1.1 in~\cite{berghlof}.
		
		We prove the other inclusion. For every function $\varphi\colon X\to\cD$ with $x\in\varphi(x)\in\cD$ and for every function $\eta\colon X\times X\to\C$ with $|\eta(x,y)|=1$ for every $(x,y)\in X\times X$, we define the operator
		\[S^{\varphi,\eta}f(x):=\frac{1}{\mu_\alpha(\varphi(x))}\sum_{y\in\varphi(x)}(f(y)-f_{\varphi(x)})\eta(y,x) q^{-\alpha|y|},\qquad f\colon X\to \C. \]
		It is easy to see that $|S^{\varphi,\eta}f|\leq M^\sharp f$ and that 
		\begin{equation}\label{supS}
			\sup_{\varphi,\eta}|S^{\varphi,\eta}f|=M^\sharp f. 
		\end{equation}
		For every $f\in (L_\alpha^{p_1},{\rm BMO}_\alpha)_{[\theta]}$, let $F\in\cF(L_\alpha^{p_1},{\rm BMO}_\alpha)$ be such that $F(\theta)=f$.
		For every $t\in\R$, by the boundedness of $M$ on $L^{p_1}$ ($p_1>1$) that we showed in Proposition~\ref{hlweak},
			\[\|S^{\varphi,\eta} (F(it))\|_{p_1}\leq\|M^\sharp(F(it))\|_{p_1} \lesssim\|M(F(it))\|_{p_1}\lesssim\|F(it)\|_{p_1}.\]
		 Furthermore, for every $t\in \R$, by the definition of the norm in ${\rm BMO}_\alpha$,
		\[\|S^{\varphi,\eta} (F(1+it))\|_{\infty}\leq\|M^\sharp(F(1+it))\|_{\infty} \leq\|F(1+it)\|_{{\rm BMO}_\alpha}.\]
		Hence $S^{\varphi,\eta}F\in\cF(L_\alpha^{p_1},L^\infty_\alpha)$ and 
		\[\|S^{\varphi,\eta}F\|_{\cF(L_\alpha^{p_1},L_\alpha^\infty)}\lesssim\|F\|_{\cF(L_\alpha^{p_1},{\rm BMO}_\alpha)}, \]
		and consequently, using $L_\alpha^p=(L_\alpha^{p_1},L_\alpha^\infty)_{[\theta]}$,
		\[\|S^{\varphi,\eta}F(\theta)\|_{p,\alpha}\lesssim \|F(\theta)\|_{(L_\alpha^{p_1},{\rm BMO}_\alpha)_{[\theta]}}. \]
		Then by passing at infimum on the set of $F\in\cF(L_\alpha^{p_1},{\rm BMO}_\alpha)$ with $F(\theta)=f$, we have that $\|S^{\varphi,\eta}F(\theta)\|_{p,\alpha}\lesssim \|f\|_{[\theta]}$.
		Finally, by Theorem~\ref{goodlambda} and by~\eqref{supS},
		\[\|f\|_{p,\alpha}\lesssim\|M^\sharp f\|_{p,\alpha}=\sup_{\varphi,\eta}\|S^{\varphi,\eta}\|_{p,\alpha}\lesssim\|f\|_{[\theta]},\]
		that proves $(L_\alpha^{p_1},{\rm BMO}_\alpha)_{[\theta]}\subseteq L_\alpha^p$.
		
		The case (ii) follows by duality.
	\end{proof}

	As a consequence of the previous theorem, we have the following interpolation result.
	\begin{corollary}\label{corintcom}
		Let $\theta\in(0,1)$ and $\frac{1}{p}=1-\theta$. Then the following hold:
		\begin{enumerate}
			\item $(L_\alpha^1,{\rm BMO}_\alpha)_{[\theta]}=L_\alpha^p$;
			\item $(H_\alpha^1,L_\alpha^\infty)_{[\theta]}=L_\alpha^p$;
			\item $(H_\alpha^1,{\rm BMO}_\alpha)_{[\theta]}=L_\alpha^p$;
		\end{enumerate}
	\end{corollary}
	\begin{proof}
		\begin{enumerate}
			\item Let $1<r<p$ and $\beta\in(0,1)$ such that $\frac{1}{r}=1-\beta+\frac{\beta}{p}$. Then we have that $(L_\alpha^1,L_\alpha^p)_{[\beta]}=L_\alpha^r$. Furthermore, by Theorem~\ref{thmintcompl}~(i), if $\gamma\in(0,1)$ is such that $\frac 1p=\frac{1-\beta}r$ we have 
			\[(L_\alpha^r,{\rm BMO}_\alpha)_{[\gamma]}=L_\alpha^p.\]
			Since $L_\alpha^1\cap{\rm BMO}_\alpha$ contains the compact supported function on $X$, it is dense in both $L_\alpha^r$ and $L_\alpha^p$. Hence, by applying Theorem~2 in~\cite{wolff}, we have that 
			\[(L_\alpha^1,{\rm BMO}_\alpha)_{[\theta]}=L_\alpha^p. \]
			\item Let $1<r<p$ and $\beta\in(0,1)$ such that $\frac{1}{r}=1-\beta+\frac{\beta}{p}$. Then we have that $(L_\alpha^r,L_\alpha^\infty)_{[\gamma]}=L_\alpha^p$, where $\gamma\in(0,1)$ is such that $\frac 1p=\frac{1-\beta}r$. Furthermore, by Theorem~\ref{thmintcompl}~(ii), we have 
			\[(H_\alpha^1,L^p_\alpha)_{[\beta]}=L_\alpha^r.\]
			Since $H_\alpha^1\cap L_\alpha^\infty$ contains the space of compact supported function on $X$ having vanishing integral, it is dense in both $L_\alpha^r$ and $L_\alpha^p$. Hence, by applying Theorem~2 in~\cite{wolff}, we have that 
			\[(H_\alpha^1,L_\alpha^\infty)_{[\theta]}=L_\alpha^p. \]
			\item It follows from~(i) and~(ii), using the inclusions $H_\alpha^1\subseteq L_\alpha^1$ and $L_\alpha^\infty\subseteq {\rm BMO}_\alpha$.
		\end{enumerate}
	\end{proof}
	\section{Boundedness of integral operators}
		Fix $\alpha>1$ and let $K\colon X\times X\to \C$ be a kernel. 
		In the doubling measure metric space $(X,\rho,\mu_\alpha)$, the standard integral H\"ormander's condition (see~\cite{hormander} and formula~(10) Ch.I in~\cite{SteinHA}) for a kernel $K\colon X\times X\to \C$ reads as follows 
		\begin{equation*}
			\sup_{v\in X,r>0}\sup_{x,y\in B_\rho(v,r)}\int_{X\setminus B_\rho(v,2r)}|K(z,x)-K(z,y)|\mu_\alpha(z)<+\infty.
		\end{equation*}
		Thanks to the shape of the balls (see~\eqref{balle}), the condition above is equivalent to
		\begin{equation}\label{hormander}
			\sup_{v\in X\setminus\{o\}}\sup_{x,y\in \sect{v}}\sum_{z\in X\setminus \sect{v}}|K(z,x)-K(z,y)|q^{-\alpha|z|}<+\infty.
		\end{equation}
		We say that $K$ satisfies \textit{H\"ormander's condition} with respect to $\mu_\alpha$ if it satisfies~\eqref{hormander}.
		
	In~Theorem~31 in~\cite{dmv} we prove that an integral operator with a kernel satisfying H\"ormander's condition that is bounded on $L^2_\alpha$ is of weak type (1,1), and then it has a bounded extension on $L_\alpha^p$ for every $p\in(1,2]$.
	
	Now we show that the same H\"ormander's condition provides a boundedness result from $H_\alpha^1$ to $L_\alpha^1$. As a consequence, given a kernel $K$ such that $K$ and $\overline{K}$ satisfy~\eqref{hormander}, if the integral operator 
	\begin{equation}\label{integraloperator}
		\cK f(z)=\sum_{x\in X}K(z,x)f(x)q^{-\alpha|x|},\qquad f\in L_\alpha^2, 
	\end{equation}
	is bounded on $L_\alpha^2$ then it is bounded on $L_\alpha^p$, for every $1<p<\infty$.
	
	\begin{theorem}\label{theobound}
		Let $K\colon X\times X\to\C$ be a kernel on the tree. Consider an integral operator $\cK$ with kernel $K$, defined as~\eqref{integraloperator},
		 that is bounded on $L_\alpha^2$. Then the following results hold.
		\begin{enumerate}
			\item If the kernel $K$ satisfies the H\"ormander's condition~\eqref{hormander}, then $\cK$ extends to a bounded linear operator from $H_\alpha^1$ to $L^1_\alpha$ and on $L^p_\alpha$, for $1<p<2$.
			\item Let $K^*(x,y):=\overline{K(y,x)}$, for every $(x,y)\in X\times X$. If $K^*$ satisfies the H\"ormander's condition~\eqref{hormander}, then $\cK$ extends to a bounded linear operator from $L_\alpha^\infty$ to ${\rm BMO}_\alpha$ and on $L^p_\alpha$, for $2<p<\infty$.
		\end{enumerate}
	\end{theorem}
	The proof of the Theorem follows from Corollary~\ref{corintcom} and it is almost verbatim the classical proof of analog results, see for example Theorem~8.2 in~\cite{cmm}.
	
	\section{Conclusions}
	We focused on the specific family of exponentially decreasing measures but the whole theory could be replicated for a larger class of measures.
	
	In~\cite{ccps} and~\cite{ccps2}, the authors introduce harmonic Bergman spaces for a  class of measures called \textit{reference measures}. A reference measure is a finite measure $\sigma\colon X\to (0,+\infty)$ that is radial with respect to $o$ and decreasing with respect to the (graph) distance from $o$. Furthermore, a reference measure $\sigma$ is \textit{optimal} if 
	\[\sup_{v\in X}\frac{\sigma(\sect{v})}{\sigma(\{v\})}<+\infty. \]
	For every $\alpha>1$, the measure $\mu_\alpha$ is an optimal reference measure.
	
	It is possible to prove that the set of reference measures that are doubling with respect to the Gromov distance is exactly the family of optimal measures satisfying 
	\[\sup_{v\in X}\frac{\sigma({p(v)})}{\sigma(\{v\})}<+\infty. \]
	
%
%
%

\end{document}